\documentclass{article}

\usepackage{amsmath,amssymb,amsfonts,amsthm,enumerate}

\usepackage{mathrsfs}

\newtheorem{thm}{Theorem}[section]
\newtheorem{lem}[thm]{Lemma}

\newtheorem{fct}[thm]{Fact}
\newtheorem{prop}[thm]{Proposition}
\newtheorem{cor}[thm]{Corollary}
\newtheorem{defn}[thm]{Definition}

\newtheorem{qus}{Question}

\def\semequiv{\kern3pt\raisebox{6pt}{\rotatebox{180}{$\cong$}}\kern3pt}

\title{Clarifying ordinals}
\author{Noah Schweber}
\date{July 2024}

\begin{document}

\maketitle

\begin{abstract} We use forcing over admissible sets to show that, for every ordinal $\alpha$ in a club $C\subset\omega_1$, there are copies of $\alpha$ such that the isomorphism between them is not computable in the join of the complete $\Pi^1_1$ set relative to each copy separately. Assuming $\mathsf{V=L}$, this is close to optimal; on the other hand, assuming large cardinals the same (and more) holds for every projective functional.
\end{abstract}

\tableofcontents

\section{Introduction}

In this paper we investigate the following general notion:

\begin{defn} Given a functional $\mathcal{F}: x\mapsto \mathcal{F}^x$ (e.g. $x\mapsto x'$, $x\mapsto\mathcal{O}^x$, $x\mapsto x^\sharp$, etc.) and a class of countable structures $\mathbb{K}$ which is closed under isomorphisms, say that $\mathcal{F}$ {\em clarifies} $\mathbb{K}$ iff for each $\mathfrak{A}\cong\mathfrak{B}\in\mathbb{K}$ each with domain $\omega$ there is an isomorphism $i:\mathfrak{A}\cong\mathfrak{B}$ which is computable from $\mathcal{F}^\mathfrak{A}\oplus\mathcal{F}^\mathfrak{B}$. 

Given a single countable structure $\mathfrak{A}$ and a functional $\mathcal{F}$, we say {\em $\mathcal{F}$ clarifies $\mathfrak{A}$} if $\mathcal{F}$ clarifies (in the sense of the above paragraph) the class of structures isomorphic to $\mathfrak{A}$.
\end{defn}

Of particular interest will be the set $\mathbb{W}$ of countable well-orderings.

For example, consider the functional $\mathcal{O}^\Box:x\mapsto\mathcal{O}^x$ sending each real to the canonical $\Pi^1_1$-complete set relative to that real (the ``$\Box$"-superscript is to distinguish this functional from the set of natural numbers $\mathcal{O}=\mathcal{O}^\emptyset$; apart from this, calligraphic letters will always denote functionals in this paper). It is easy to see that if $\mathfrak{A},\mathfrak{B}$ are isomorphic structures with domain $\omega$, then --- conflating structures with domain $\omega$ and their atomic diagrams, as usual --- an isomorphism between them is computable from $\mathcal{O}^{\mathfrak{A}\oplus\mathfrak{B}}$, but this does {\em not} extend to showing that such an isomorphism can be computed from merely $\mathcal{O}^\mathfrak{A}\oplus\mathcal{O}^\mathfrak{B}$. 

However, in many cases this weaker computational power does in fact suffice: for instance, if $\alpha$ is a computable ordinal and $\mathfrak{A}\cong\mathfrak{B}\cong\alpha$ are copies with domain $\omega$ (perhaps not computable themselves), then the isomorphism between them is computable in $\mathcal{O}^\mathfrak{A}\oplus\mathcal{O}^\mathfrak{B}$. Essentially, $\mathcal{O}^\mathfrak{A}$ can compute the isomorphism between $\mathfrak{A}$ and the lowest-index computable well-ordering of $\omega$, and $\mathcal{O}^\mathfrak{B}$ can do the same for $\mathfrak{B}$; but these two lowest-index computable well-orderings are in fact equal, so we can compose one isomorphism with the inverse of the other to get the isomorphism $\mathfrak{A}\cong\mathfrak{B}$.  This raises the question of whether $\mathcal{O}^\Box$ (or a similar functional) clarifies all, or even ``most," countable ordinals.\footnote{This problem is related in spirit to the notions of {\em degrees of categoricity} and {\em classifiability of equivalence relations by (Borel) invariants}, but to the best of my knowledge has not specifically been treated before. However, the interested reader may consult Anderson/Csima \cite{AnCs16} and Hjorth/Kechris \cite{HjKe97}, respectively, for background on these topics.} 

We first show that the functional version of Kleene's $\mathcal{O}$ $$\mathcal{O}^\Box: x\mapsto\mathcal{O}^x$$ does {\em not} clarify $\mathbb{W}$. Our proof uses set-theoretic arguments and so gives a quite large upper bound on the first non-$\mathcal{O}^\Box$-clarified ordinal; we partially explain this by showing that the smallest $\mathcal{O}^\Box$-unclarifiable ordinal must be quite large (although there is still a gap between lower and upper bounds). We also explore the role of large cardinal hypotheses in analogues of these results for general projective functionals in place of $\mathcal{O}^\Box$ specifically. We end by presenting some open problems.

Apart from the functional notation described above, our notation throughout is standard. For background on admissible sets, we refer to Barwise \cite{Barw75} for general facts and to Ershov \cite{Ersh90} for the result (independently proved by Barwise) that set forcing preserves admissibility. We refer to Ash/Knight \cite{AsKn00} for background on computable structure theory.






\section{$\mathcal{O}^\Box$ does not clarify $\mathbb{W}$}

Recall that $\mathbb{W}$ is the set of countable well-orderings.

\begin{defn} For $n\in\omega$, an ordinal $\alpha$ is {\em $n$-potent} iff $L_{\alpha^\boxplus}\models$ ``There is no $\Sigma_n(L_{\alpha^\boxplus})$ injection from $\alpha$ to $\omega$."
\end{defn}
 
\begin{lem} There is a club of countable $n$-potent ordinals, for each $n$.
\end{lem}

\begin{proof} For each countable elementary $\mathfrak{M}\prec L_{\omega_1^\boxplus}$, let $m_\mathfrak{M}:\mathfrak{M}\cong L_{\mu_\mathfrak{M}}$ be the Mostowski collapse map. The set of $m_\mathfrak{M}$-images of $\omega_1$ as $\mathfrak{M}$ ranges over countable elementary substructures of $L_{\omega_1^\boxplus}$ is the desired club.
\end{proof}

\begin{thm}\label{upperbound} If $\alpha<\omega_1$ is $3$-potent, then $\alpha$ is not $\mathcal{O}^\Box$-clarified.
\end{thm}

\begin{proof} Let $\alpha$ be $3$-potent. We start by defining the forcing we will use. Let $\mathbb{P}=Col(\omega,\alpha)\times Col(\omega,\alpha)$ (note that conditions in $\mathbb{P}$ are finite objects and so there is no need to relativize this definition to $L_{\alpha^\boxplus}$), and let $g,g_0,g_1$ be the canonical names for the generic filter (viewed as a pair of mutually-$Col(\omega,\alpha)$-generics), the left coordinate of the generic filter, and the right coordinate of the generic filter, respectively. Fix $G$ $\mathbb{P}$-generic over $L_{\alpha^\boxplus}$ (which exists since $\alpha$ and hence $L_{\alpha^\boxplus}$ is countable); we claim that the pair of copies of $\alpha$ corresponding to $G_0:=g_0[G]$ and $G_1:=g_1[G]$ --- which we will conflate with $G_0$ and $G_1$ themselves --- will have the desired property that $\mathcal{O}^{G_0}\oplus\mathcal{O}^{G_1}$ does not compute the isomorphism between $G_0$ and $G_1$. To prove this, we need to combine a complexity calculation and a counting argument. 

First, note that if $A$ is an admissible set\footnote{That is, a transitive set model of $\mathsf{KP}$; for our purposes, we only care about admissible levels of $L$ and their generic extensions.} and $X\in A$ is a linear order then $X$ is well-founded in reality iff $A$ satisfies the $\Sigma_1$ sentence ``There is an order-isomorphism between $X$ and an ordinal." 

Now we ``genericize" the previous paragraph. A condition $(p,q)$ forces that $n\in\mathcal{O}^{G_i}$ for $i\in\{0,1\}$ iff there is an extension $(p',q')$ and a name $\mu$ which $(p',q')$ forces to be an order-isomorphism between $\Phi_n^{G_i}$ and some ordinal, which is a $\Sigma_3(L_{\alpha^\boxplus})$ property. For brevity, going forwards we will simply write ``$\Sigma_3$."

Finally, we can apply a counting argument using $3$-potency of $\alpha$. Suppose towards contradiction that $\mathcal{O}^{G_0}\oplus\mathcal{O}^{G_1}$ does compute the isomorphism between $G_0$ and $G_1$. Let $(p,q)\in\mathbb{P}, e\in\omega$ be such that $$(p,q)\Vdash_\mathbb{P}\Phi_e^{\mathcal{O}^{g_0}\oplus\mathcal{O}^{g_1}}: g_0\cong g_1.$$ By extending and permuting (the latter at the cost of changing $e$) the two coordinates as necessary, we may without loss of generality assume $p=q$. Let $n=\vert p\vert$, and let $E=\alpha\setminus ran(p)$ be the set of ``unused" elements of $\alpha$; we still have $L_{\alpha^\boxplus}\models\vert E\vert=\omega_1$ since $p$ is finite.

Now by the argument two paragraphs prior, and by picking an appropriately-least witness, we have a $\Sigma_3$ map assigning to each $\eta\in E$ a pair of finite strings of natural numbers $(\sigma_\eta,\tau_\eta)$ such that for some condition $$(p_\eta,q_\eta)\le (p^\smallfrown\langle\eta\rangle, p^\smallfrown\langle\eta\rangle)$$ we have $(p',q')\Vdash \sigma_\eta\prec \mathcal{O}^{g_0}\wedge \tau_eta\prec\mathcal{O}^{g_1}$ and moreover $\Phi_e^{\sigma_\eta\oplus\tau_\eta}(n)\downarrow=n$ (that is, the putative isomorphism computed by $\Phi_e$ identifies the ``two $\eta$s"). By the above assumptions on $e$ and $p$, at least one such pair of strings must exist, so since $\alpha$ is $3$-potent this map is not injective.

Let $\zeta<\xi<\alpha$ be such that $\sigma_\zeta=\sigma_\xi$ and $\tau_\zeta=\tau_xi$, with corresponding conditions $(p_\zeta,q_\zeta)$ and $(p_\xi,q_\xi)$. The condition $(p_\zeta, q_\xi)$ then forces the putative isomorphism computed by $\Phi_e$ to identify $\zeta$ and $\xi$, a contradiction.
\end{proof}

Combining the theorem and lemma above, we get that ``most" countable ordinals are not $\mathcal{O}^\Box$-classifiable:

\begin{cor}\label{nonstat} The subset of $\mathbb{W}$ consisting of countable ordinals which are clarified by $\mathcal{O}^\Box$ is nonstationary.
\end{cor}








\section{Positive results nonetheless}

In this section we present results demonstrating the limits of the previous section. First, we examine the situation assuming $\mathsf{V=L}$, and show in particular that Corollary \ref{nonstat} above is best possible in $\mathsf{ZFC}$. Then we show in $\mathsf{ZFC}$ that the smallest non-$\mathcal{O}^\Box$-clarified ordinal is quite large. Since these results use different techniques, we treat them in separate subsections.

\subsection{Clarification below $\omega_1^L$}

First, we observe that ``nonstationary" cannot be replaced with ``countable" in Corollary \ref{nonstat} in $\mathsf{ZFC}$ alone:

\begin{thm}\label{unbound} The set of ordinals clarified by $\mathcal{O}^\Box$ is unbounded in $\omega_1^L$.
\end{thm}

Note that it is consistent with $\mathsf{ZFC}$ that $\omega_1=\omega_1^L$.

\begin{proof} It is known that there are unbounded-in-$\omega_1^L$-many ordinals $\alpha$ such that $L_{\alpha+\omega}$ (or even $L_{\alpha+1}$) sees that $\alpha$ is countable; for example, this is the case for any $\alpha$ such that $L_\alpha$ is a pointwise-definable model of a large enough fragment of set theory (see \cite{HLR13}). Se will show that every such $\alpha$ is clarified by $\mathcal{O}^\Box$.

Suppose $\alpha$ is such an ordinal. Given a copy $\mathfrak{X}$ of $\alpha$ with domain $\omega$, there is a unique-up-to-isomorphism countable structure $\mathcal{M}$ satisfying $\mathsf{Pairing}$, $\mathsf{Union}$, $\mathsf{Foundation}$, $\mathsf{Extensionality}$, and $\mathsf{V=L}$ such that $\mathit{Ord}^\mathcal{M}\cong\alpha+\omega$. Such an $\mathcal{M}$ is necessarily well-founded, so we can safely conflate $\mathbb{R}^\mathcal{M}$ with a set of reals. Under this conflation, let $r=r(\mathfrak{X})$ be the unique real which some(/every) such $\mathcal{M}$ thinks is ``the $<_L$-least real coding a copy of the largest limit ordinal" (note that in $\mathcal{M}$, the largest limit ordinal is $\alpha$).

We first claim that $r\le_T\mathcal{O}^\mathfrak{X}$. To see this, note that for $n\in\omega$ we have $n\in r$ iff for every countable $\mathcal{M}$ with the above properties, $n$ is in the real which $\mathcal{M}$ thinks is $r$; the predicate $n\not\in r$ has a similarly $\Pi^1_1(\mathfrak{X})$ definition, gotten by replacing ``for every" with ``for some" in the previous sentence, and so $r\le_T\mathcal{O}^\mathfrak{X}$ as desired.

Next, we claim that $\mathcal{O}^\mathfrak{X}$ also computes the unique isomorphism $i_\mathfrak{X}$ from $\mathfrak{X}$ to the well-ordering coded by $r$. Again, the point is that that isomorphism is an element of every $\mathcal{M}$ with the properties above, so $\mathcal{O}^\mathfrak{X}$ can determine the isomorphism's behavior on each $x\in\mathfrak{X}$.

Now suppose $\mathfrak{X},\mathfrak{Y}$ are copies with domain $\omega$ of the same countable ordinal $\alpha$. Note that $r(\mathfrak{X})=r(\mathfrak{Y})$, and so in particular $i_\mathfrak{X}\circ i_\mathfrak{Y}^{-1}$ is an isomorphism from $\mathfrak{X}$ to $\mathfrak{Y}$. But that isomorphism is clearly computable from the join of $i_\mathfrak{X}$ and $i_\mathfrak{Y}$, which in turn is computable from $\mathcal{O}^\mathfrak{X}\oplus\mathcal{O}^\mathfrak{Y}$ by the discussion above.
\end{proof}

Since it is consistent with $\mathsf{ZFC}$ that $\omega_1^L=\omega_1$, we immediately have the following:

\begin{cor}\label{unboundcor} It is consistent with $\mathsf{ZFC}$ that $\mathcal{O}^\Box$ clarifies an uncountable set of countable ordinals.
\end{cor}

A similar phenomenon occurs if we shift from $\mathcal{O}^\Box$ to more complicated functionals; in contrast with Theorem \ref{unbound}, we find it simpler to state this theorem assuming $\mathsf{V=L}$. Note that in what follows, we gauge the complexity of a functional as a set of ordered pairs of reals; so, for example, $\mathcal{O}^\Box$ is $\Pi^1_1\wedge\Sigma^1_1$, even though $\mathcal{O}^x$ as a set of natural numbers is uniformly $\Pi^1_1$ in $x$ for each real $x$. (This is similar to how e.g. ${\bf 0'''}$ is a $\Pi^0_2$ singleton but a $\Sigma^0_3$ set of naturals.)

\begin{thm}\label{constructible} Assume $\mathsf{V=L}$. Let $\mathcal{Q}^\Box$ be the  functional such that:\begin{itemize}

\item If $r$ does not code a well-ordering, then $\mathcal{Q}^r=0$.

\item If $r$ does code a well-ordering, then $\mathcal{Q}^r$ is the $L$-least real $s$ which codes some well-founded model of $\mathsf{KP+V=L}$ containing $r$ as an element and containing an isomorphism between (the structure coded by) $r$ and some ordinal.
\end{itemize}

Then $\mathcal{Q}^\Box$ is $\Sigma^1_2\wedge\Pi^1_2$ and clarifies every ordinal $<\omega_1^L$.
\end{thm}

\begin{proof} First we argue that $\mathcal{Q}^\Box$ is $\Sigma^1_2\wedge\Pi^1_2$. Determining which bulletpoint holds of a given $r$ is a $\Pi^1_1$ question, so we may assume that $r$ does in fact code a well-ordering. The existence of a well-founded model of $\mathsf{KP+V=L}$ with the relevant properties is $\Sigma^1_2$, and the $L$-ordering is $\Delta^1_2$. So (assuming $r$ codes a well-ordering) we have $\mathcal{Q}^r=s$ iff a $\Sigma^1_2$ condition holds and for every $t$ a disjunction between a $\Delta^1_2$ condition and a $\Pi^1_2$ condition holds; and this simplifies to $\Sigma^1_2\wedge\Pi^1_2$.

Next, we show that $\mathcal{Q}^\Box$ does indeed clarify the $L$-countable ordinals. Let $\mathfrak{X},\mathfrak{Y}$ be copies of the same countable ordinal $\alpha$. Let $L_\beta, L_\gamma$ be the levels of $L$ coded by $\mathcal{Q}^\mathfrak{X},\mathcal{Q}^\mathfrak{Y}$ respectively, with $f\in L_\beta$ and $g\in L_\gamma$ being the corresponding isomorphisms-with-$\alpha$. The key point is that levels of the $L$-hierarchy are comparable. Without loss of generality, assume $\beta\le\gamma$. Then $\mathfrak{X},f\in L_\gamma$ and so in particular the real $s=\mathcal{Q}^\mathfrak{Y}$ coding $L_\gamma$ can compute $g^{-1}\circ f:\mathfrak{X}\cong\mathfrak{Y}$.
\end{proof}

\subsection{The smallest nonclarifiable ordinal is large}

\begin{defn} A transitive set $A=(A;\in)$ is \begin{itemize}

\item {\em pointwise-definable} iff for each $a\in A$ there is a parameter-free formula in the language of set theory $\varphi_a$ such that $\varphi_a^A=\{a\}$.

\item $\Sigma^1_1$-pointwise-definable iff for every $a\in L_\alpha$ there is a $\Sigma_1^1$ formula in the language of set theory $\varphi_a(x)\equiv\exists X\psi(X,x)$ such that $\varphi_a^{L_\alpha}=\{a\}$. (Note that the ``$\exists X$" is understood as a true second-order quantifier here, and in particular $X$ is not restricted to definable subsets of $L_\alpha$.)
\end{itemize}

An {\em ordinal} $\alpha$ is pd (resp. $\Sigma^1_1$-pd) iff $L_\alpha$ is pointwise-definable (resp. $\Sigma^1_1$-pointwise-definable).
\end{defn}

For more on pointwise-definability, see Hamkins/Linetsky/Reitz \cite{HLR13} (where this term was introduced).

In this subsection we prove:

\begin{thm}\label{lowerbound} If $\alpha$ is $\Sigma^1_1$-pd, then $\mathcal{O}^\Box$ clarifies $\{\mathfrak{A}: \mathfrak{A}\cong\alpha\}$.
\end{thm}

Before proving the theorem, we will say a bit more about $\Sigma^1_1$-pd-ness. We begin by mentioning Farmer Schlutzenberg's result (see Schlutzenberg \cite{Schl21} or Schweber \cite{Schw??}) that the smallest non-$\Sigma^1_1$-pointwise-definable ordinal is precisely $(\omega_1)^{L_\theta}$, where $L_\theta$ is the smallest level of the $L$-hierarchy satisfying $\mathsf{KP}+$ ``$\omega_1$ exists." Note that this gives an easy proof that there are non-$\Sigma^1_1$-pd ordinals $<\omega_1^L$, and so Theorem \ref{lowerbound} does not imply Theorem \ref{unbound}. It is easy to see that this is much larger than more computability-theoretically-familiar ordinals such as $\beta_0$. 

It is also worth noting that the ``$\Sigma^1_1$" prefix is non-redundant:

\begin{prop} The least non-$\Sigma^1_1$-pd ordinal is strictly larger than the least non-pd ordinal.
\end{prop}

\begin{proof} We use Schlutzenberg's characterization above for simplicity. Since the elementary diagram of $L_{\omega_1^{L_\theta}}$ is an element of $L_\theta$, by a counting argument in $L_\theta$ we have that $L_{\omega_1^{L_\theta}}$ contains an undefinable element $c$. Again inside $L_\theta$, apply downward Lowenheim-Skolem to get a countable elementary submodel of $L_{\omega_1^{L_\theta}}$ containing $c$ as an element and apply the Mostowski collapse; by condensation and (internal) countability the result is some $L_\gamma$ for $\gamma<\omega_1^{L_\theta}$, and by elementarity that $L_\gamma$ must not be pointwise-definable in the first-order sense.
\end{proof}

The above shows that the upper bound provided in the previous section is perhaps not as silly as it may appear.

\begin{proof}[Proof of Theorem \ref{lowerbound}] Given a well-ordering $\mathfrak{X}$ with domain $\omega$ which is isomorphic to an ordinal $\alpha$, let a {\em annotation} of $\mathfrak{X}$ be any structure  $\widehat{\mathfrak{X}}$ which is an expansion of $L_\alpha$ with the following properties:\begin{itemize}

\item The underlying set of $\widehat{\mathfrak{X}}$ is $\omega$.

\item $\mathit{Ord}^{\widehat{\mathfrak{X}}}$ is the set $\{2n:n\in\omega\}$ of even numbers, and $n\mapsto 2n$ gives an isomorphism from $\mathfrak{X}$ to $(\mathit{Ord}^{\widehat{\mathfrak{X}}};\in^{\widehat{\mathfrak{X}}}\upharpoonright \mathit{Ord}^{\widehat{\mathfrak{X}}})$.

\item The additional structure on $\widehat{\mathfrak{X}}$ consists of $(i)$ for each $\Sigma^1_1$ formula $\varphi(x)\equiv\exists Y\psi(x,Y)$ in the language of set theory, a unary predicate $U_\varphi$ picking out exactly $\varphi^{\widehat{\mathfrak{X}}}$ and $(ii)$ a nullary predicate $V_\varphi$ which evaluates to $\top$ iff $\vert\varphi^{\widehat{\mathfrak{X}}}\vert=1$.
\end{itemize}

(Note that while many annotations of a given $\mathfrak{X}$ exist, they are all isomorphic.) Crucially, $\mathcal{O}^\mathfrak{X}$ can compute an annotation of $\mathfrak{X}$, since $\mathcal{O}^\mathfrak{X}$ is $\Pi^1_1(\mathfrak{X})$-complete: we can pass from $\mathfrak{X}$ to a copy of $L_\alpha$ in a $\Delta^1_1$ (in $\mathfrak{X}$) way, and the further structure is all determined by Boolean combinations of $\Pi^1_1$ questions about that copy of $L_\alpha$.

Now let $\alpha$ be $\Sigma^1_1$-pointwise-definable and let $\mathfrak{A}\cong\mathfrak{B}$ be copies of $\alpha$. Suppose first that we are also provided with annotations $\widehat{\mathfrak{A}}$ and $\widehat{\mathfrak{B}}$ respectively. Since $\alpha$ is $\Sigma^1_1$-pointwise-definable, an isomorphism $\widehat{\mathfrak{A}}\cong\widehat{\mathfrak{B}}$ can be computed from the join of (the atomic diagrams of) $\widehat{\mathfrak{A}}$ and $\widehat{\mathfrak{B}}$ as follows: in order to tell where to map $a\in\widehat{\mathfrak{A}}$, we first look for a $\varphi$ such that $\mathfrak{A}\models U_\varphi(a)\wedge V_\varphi$ and then map this $a$ to the unique $b\in\mathfrak{B}$ such that $\mathfrak{B}\models U_\varphi(b)$. This isomorphism restricts to an isomorphism between their ordinal parts, which in turn yields one between the original $\mathfrak{A}$ and $\mathfrak{B}$ themselves. Concretely, given $j:\widehat{\mathfrak{A}}\cong\widehat{\mathfrak{B}}$, we let $i:\mathfrak{A}\cong\mathfrak{B}: a\mapsto {j(2a)\over 2}$.
\end{proof}








\section{Large cardinals and projective functionals}

In this section we show that the clarification results of Section 3.1 above were ``set-theoretically fragile:" if we assume large cardinal hypotheses in place of $\mathsf{V=L}$, we get a radically different picture. For simplicity we focus on functionals which are in $L(\mathbb{R})$ (note that every projective functional is in $L(\mathbb{R})$), but by modifying the relevant large cardinal hypothesis we can achieve analogous results for wider or narrower classes of functionals.

\bigskip

We begin with a lemma which does not need large cardinals at all, and is essentially the ``de-$L$-ification" of Theorem \ref{upperbound} above.

\begin{lem}\label{countZFC} Let $\mathcal{Q}$ be a projective functional. Suppose $\alpha\ge (2^{\aleph_0})^+$ and $\mathbb{P}$ is a (set) forcing such that there are $\mathbb{P}$-names $\mu,\nu$ which are forced by $\mathbb{P}$ to yield a pair of $\omega$-copies of $\alpha$ which are mutually $Col(\omega,\alpha)$-generic over the ground model $V$.

Then if $G$ is $\mathbb{P}$-generic over $V$, the intermediate model $V[\mu[G],\nu[G]]$ satisfies $\mathsf{ZFC}$ + ``$\mathcal{Q}$ does not clarify $\alpha$."
\end{lem}

\begin{proof} This is the same counting argument as in the proof of Theorem \ref{upperbound}, with $\mathcal{Q}$ replacing $\mathcal{O}^\Box$ and with no need to pay attention to quantifier complexity.
\end{proof}

Next, we recall three useful results on forcing and large cardinals. The first is Grigorieff's intermediate model theorem:

\begin{fct}[Grigorieff \cite{Grig75}]\label{Grigorieff} If $M\subseteq N\subseteq M'$ are models of $\mathsf{ZFC}$ and $M'$ is a (set) generic extension of $M$, then $M'$ is also a generic extension of $N$.
\end{fct}

The second useful result is that ``small" forcing does not destroy large cardinals. This was first proved for measurable cardinals by Levy-Solovay, but the same proof works for Woodin cardinals:\footnote{There is some care needed here. Levy-Solovay also proved that ``small" forcing does not {\bf create} large cardinals either, which is harder to prove. In particular, while non-creation also holds for Woodin cardinals, it requires a different proof; see Hamkins/Woodin \cite{HaWo00}. Here we will only need the much simpler ``non-destruction" part of the theorem, and this much seems fair to attribute to the original Levy-Solovay paper.}

\begin{fct}[Levy-Solovay \cite{LeSo67}]\label{indest} If $V\models$ ``There is a proper class of Woodin cardinals," then every set-generic extension of $V$ satisfies that as well.
\end{fct}

Finally, we will need Woodin's generic absoluteness result:

\begin{fct}[Woodin]\label{abso} Suppose there is a proper class of Woodin cardinals. If $V[G]$ is a set-generic extension of $V$, $L(\mathbb{R})^V\equiv L(\mathbb{R})^{V[G]}$ and so in particular every projective sentence with real parameters from $V$ is absolute between $V$ and $V[G]$.
\end{fct}

Woodin's original proof of this theorem is unpublished, but proofs can be found in Larson \cite{Lars04} (Theorem 3.1.12) or in Steel \cite{Stee10} (Theorem 7.22). (In fact, Woodin proved that in a further generic extension there is an elementary embedding $L(\mathbb{R})\rightarrow L(\mathbb{R})^{V[G]}$.) Separately, it is worth noting that $L(\mathbb{R})^{V[G]}\not=L(\mathbb{R})[G]$ in general even for forcings living in $L(\mathbb{R})$.

\bigskip

We can put the results above together as follows:

\begin{thm}\label{noproj} Suppose there is a proper class of Woodin cardinals. Then no projective functional $\mathcal{Q}$ clarfies uncountably many ordinals.
\end{thm}

\begin{proof} Let $\mathcal{Q}$ be a projective functional. By Fact \ref{abso}, any two projective definitions of $\mathcal{Q}$ in $V$ will continue to define the same functional in any generic extension of $V$, so we may freely conflate $\mathcal{Q}$ with some fixed definition. In particular, we can (pretend to) apply $\mathcal{Q}$ to reals in generic extensions of $V$.

Let $\kappa$ be the least inaccessible. Let $\mathbb{P}$ be the ``full" Levy collapse of everything below $\kappa$ to $\omega$; precisely, $\mathbb{P}$ is the finite support product over $\alpha\in[\omega,\kappa)$ of $Col(\omega,\alpha)$. Forcing with $\mathbb{P}$ yields for each infinite $\alpha<\kappa$ a pair of mutually $Col(\omega,\alpha)$-generics over $V$ (e.g. from  the $Col(\omega,\alpha)\times Col(\omega,\alpha+1)$ factor). By Lemma \ref{countZFC}, if $V[G]$ is the full generic extension by $\mathbb{P}$ then we have for each $\omega\le\alpha<\kappa$ an intermediate model $M_\alpha\models\mathsf{ZFC}$ between $V$ and $V[G]$ such that $M_\alpha\models$ ``$\mathcal{Q}$ does not clarify $\alpha$."

By Fact \ref{Grigorieff}, $V[G]$ is a generic extension of $M_\alpha$. By Fact \ref{indest}, the $M_\alpha$s still satisfy our large cardinal hypotheses. Since ``$\mathcal{Q}$ clarifies $\alpha$" is projective as long as $\mathcal{Q}$ is, this lets us apply Fact \ref{abso} twice. First, for each $\alpha$ we ``transfer upwards" the sentence ``$\alpha$ is not $\mathcal{Q}$-clarified" from $M_\alpha$ to $V[G]$. Second, we go downwards from $V[G]$ to $V$, this time transferring the sentence ``Only countably many countable ordinals are $\mathcal{Q}$-clarified" (which holds in $V[G]$ since $\kappa=\omega_1^{V[G]}$ by standard results about the Levy collapse).
\end{proof}








\section{Open questions}
We close by mentioning a couple directions for further research.

\bigskip

The most immediate open problem is simply to close the gap between the lower and upper bounds on the smallest non-$\mathcal{O}$-clarifiable ordinal provided in this paper:

\begin{qus} What is the least ordinal $\alpha$ witnessing that $\mathcal{O}^\Box$ does not clarify $\mathbb{W}$?
\end{qus}

Note that Theorems \ref{upperbound} and \ref{lowerbound} together parallel the upper- and lower-bound results from Schweber \cite{Schw??} approximating the smallest ordinal $\alpha$ such that the supremum of the ordinals Medvedev-reducible to $\alpha$ is strictly smaller than the next admissible above $\alpha$. While these two questions do not appear directly related, this situation does suggest that the same techniques may be relevant for their resolution.

\bigskip

A second question, along similar lines, is the following: 

\begin{qus} Suppose that there is a proper class of Woodin cardinals. For a given projective functional $\mathcal{Q}\in\Pi^1_n\wedge\Sigma^1_n$, can we provide a reasonable upper bound on the first $\mathcal{Q}$-unclarifiable ordinal?
\end{qus}

Technically an upper bound can be extracted from the proof of Theorem \ref{noproj}, but it is truly terrible. We can adapt the argument of Theorem \ref{upperbound} to partially address this, but the result is still not very satisfying: the issue is that we need, in place of ``admissible set," an analogous notion such that $(i)$ each set of this type computes $\Pi^1_n\wedge\Sigma^1_n$ facts in a globally $\Sigma_{n-1}\vee\Pi_{n-1}$ way and $(ii)$ the class of such sets is closed under set forcing (and a couple other conditions besides). However, there does not appear to yet be such a notion which is as structurally well-understood as admissibility.

\bigskip

Finally, there is also a related question on continuous reducibility (cf. Pauly \cite{Paul10} for background on that notion):

\begin{qus} Suppose $\mathcal{F}$ clarifies $\mathbb{W}$. Must $\mathcal{O}^\Box$ be continuously reducible to $\mathcal{F}$?
\end{qus}

That is, if $\mathcal{F}$ clarifies $\mathbb{W}$ must there be a pair of continuous functionals $\mathcal{R}_0,\mathcal{R}_1$ such that $\mathcal{R}_1\circ\mathcal{F}\circ\mathcal{R}_0=\mathcal{O}$? There are also a number of variants of continuous reducibility which have been studied, and the analogous question for any such variant is also interesting. I suspect that the answer to the question is affirmative (and this paper grew out of an attempt to prove that suspicion), but at the moment I do not even know whether there must be a {\em Borel} reduction of $\mathcal{O}^\Box$ to any such $\mathcal{F}$.







\end{document}